\documentclass[preprint,12pt]{elsarticle}
\usepackage{latexsym}
\usepackage{amssymb}
\usepackage{amsmath}
\usepackage{amsthm}
\usepackage{verbatim}
\usepackage[pagewise]{lineno}
\numberwithin{equation}{section}
\newtheorem{theorem}{Theorem}[section]
\newtheorem{lemma}{Lemma}[section]

\newtheorem{definition}{Definition}[section]

\allowdisplaybreaks
\usepackage{geometry}

\biboptions{sort&compress}

\usepackage{lineno}
\begin{document}
\begin{frontmatter}



\title{Analysis on a generalized two-component Novikov system
}

\author[ad1]{Yonghui Zhou}
\ead{zhouyhmath@163.com}
\author[ad2]{Xiaowan Li}
\ead{xiaowan0207@163.com}
\author[ad3]{Shuguan Ji\corref{cor}}
\ead{jisg100@nenu.edu.cn}
\author[ad4]{Zhijun Qiao}
\ead{zhijun.qiao@utrgv.edu}
\address[ad1]{School of Mathematics, Hexi University, Zhangye 734000, P.R. China}
\address[ad2]{College of Mathematics and System Sciences, Xinjiang University, Urumqi, 830046, P.R. China}
\address[ad3]{School of Mathematics and Statistics and Center for Mathematics and Interdisciplinary Sciences, Northeast Normal University, Changchun 130024, P.R. China}
\address[ad4]{School of Mathematical and Statistical Sciences, The University of Texas Rio Grande Valley, Edinburg, Texas 78539, USA}
\cortext[cor]{Corresponding author.}

\begin{abstract}
In this paper, we study the Cauchy problem for a generalized two-component Novikov system with weak dissipation. We first establish the local well-posedness of solutions by using the Kato's theorem. Then we give the necessary and sufficient condition for the occurrence of wave breaking in a finite time. Finally, we investigate the persistence properties of strong solutions in the weighted $L^{p}(\mathbb{R})$ spaces for a large class of moderate weights.
\end{abstract}


\end{frontmatter}


\section{Introduction}

In this paper, we study 
the Cauchy problem for a generalized two-component Novikov system with weak dissipation
\begin{equation}\begin{cases}
u_{t}-u_{txx}+\left(g(u)\right)_{x}+\lambda(u-u_{xx})
=k\left(3uu_{x}u_{xx}+u^{2}u_{xxx}\right)+k \rho(u\rho)_{x},\\
\rho_{t}+k u^{2}\rho_{x}+k \rho uu_{x}=0,
\end{cases}\label{101}
\end{equation}
with the initial data
\begin{equation}
u(0,x)=u_{0}(x),\ \rho(0,x)=\rho_{0}(x),
\label{102}
\end{equation}
where $g(u)\in C^{\infty}(\mathbb{R},\mathbb{R}), g(0)=0$, $k\in \mathbb{R}\setminus\{0\}$, $\lambda$ is a nonnegative dissipative parameter.

For $g(u)=\frac{4}{3}u^{3}$, $\lambda=0$ and $k=1$, \eqref{101} corresponds to the following two-component Novikov system
\begin{equation}
\begin{cases}
u_{t}-u_{txx}+4u^{2}u_{x}
=3uu_{x}u_{xx}+u^{2}u_{xxx}+\rho(u\rho)_{x},\\
\rho_{t}+u^{2}\rho_{x}+\rho uu_{x}=0.
\end{cases}\label{103}
\end{equation}
Such a model was first proposed by Popowicz \cite{Popowicz2015}, in which the author verified that system \eqref{103} has a Hamiltonian structure. 
In particular, 
if $\rho=0$, then \eqref{103} exactly becomes the Novikov equation \cite{Novikov2009} 
\begin{equation}
u_{t}-u_{txx}+4u^{2}u_{x}=3uu_{x}u_{xx}+u^{2}u_{xxx}
\label{104}
\end{equation}
which 
is integrable and  possesses similar properties to the Camassa-Holm equation and the Degasperis-Procesi
equation. 
It is worth mentioning that \eqref{104} has a bi-Hamiltonian structure \cite{Hone2009}, complete integrability,  an infinite sequence of conserved quantities, and peaked solitons (peakons) \cite{Hone2008}. 
The initial impetus for investigating the Novikov equation is that it can be regarded as a generalization of the celebrated Camassa-Holm (CH) equation
\begin{equation}
u_{t}-u_{txx}+3uu_{x}=3uu_{x}u_{xx}+u^{2}u_{xxx}.
\label{105}
\end{equation}
Such model is a integrable equation which describes the unidirectional propagation of shallow water waves over a flat bottom. The CH equation was  implied in 
Fokas and Fuchssteiner \cite{Fokas1981} as a special member in the hierarchy of evolution equation determined through the recursion operator scheme and  bi-Hamiltonian structures 
and  became famous due to the shallow water wave regime and peakon feature propsoed by Camassa and Holm \cite{Camassa1993} from physical principles.
The significant differences between the CH equation and the Novikov equation are: 1) the former has quadratic nonlinearity while the latter has cubic nonlinearity; 2 ) Lax pair for the former is 2 by 2 while the latter is 3 by 3; and 3) the $N$-peakon dynamical system is completely integrable for the former, but not sure yet for the latter.


For $\rho\equiv0$, \eqref{101} corresponds to the following weakly dissipative generalized Novikov equation \cite{Ji2021}
\begin{equation}
u_{t}-u_{txx}+\left(g(u)\right)_{x}+\lambda(u-u_{xx})
=k\left(3uu_{x}u_{xx}+u^{2}u_{xxx}\right).
\label{106}
\end{equation}
If $g(u)=\frac{4}{3}u^{3}$ and $k=1$, then \eqref{106} is reduced to the weakly dissipative Novikov equation \cite{Yan2012}
\begin{equation}
u_{t}-u_{txx}+4u^{2}u_{x}+\lambda(u-u_{xx})=3uu_{x}u_{xx}+u^{2}u_{xxx}.
\label{107}
\end{equation}

In the last 30 years, the CH equation and its various generalizations
were tremendously investigated, such as the local well-posedness, global existence of strong solutions, persistence property, wave breaking phenomena, global existence of peakons, and global existence of conservative weak solutions,
see \cite{Constantin19981,Constantin1998,Constantin2000,Qiao2003CMP,Qiao2004AAM,Bressan2007,Chen2011,Brandolese2012,Brandolese2014,Brandolese20141,Brandolese20142,
Constantin20004,Constantin2009,Wahlen2006,Qiao2006JMP,Luo2015,Mustafa20072,Freire1,Freire2,Freire3,Ji20221,Chen2024,Chen2023,Qiao2024,Zhou2025,Chen2018,Ni2011,
Gui2010,Ji2021,Ji2022,Novruzov2022,Coclite20052,Zhou2022,Zhou2024} and the reference therein.
Inspired by the previous work, the goal of this paper is to study the local well-posedness, wave breaking phenomena and persistence properties of strong solutions for system \eqref{101}.


\textbf{Notation.} Throughout this paper, all spaces of functions are over $\mathbb{R}$ and for simplicity, we drop $\mathbb{R}$ in our notation of function spaces if there is no ambiguity. Additionally, we denote by $\|\cdot\|_{s}$ the norm in the Sobolev space $H^{s}(\mathbb{R})$. $C>0$ denotes a generic constant, $C_{i}>0(i=1,2,...)$ represents a specific constant.

\section{Local well-posedness}

\setcounter{equation}{0}

\label{sec:2}

In this section, we prove the local well-posedness result by using Kato's theorem.
For convenience, we state Kato's theorem in the form suitable for our purpose.
Consider the abstract quasilinear evolution equation of the form
\begin{equation}
\frac{dz}{dt}+A(z)z=f(z),\ t>0,
\label{201}
\end{equation}
with the initial datum $ z(0)=z_{0}.$
Assume that:
\begin{description}
\item[$(i)$] $A(y)\in L(Y, X)$ for $y\in Y$ with
$$\|(A(y)-A(z))w\|_{X}\leq \rho_{1}\|y-z\|_{X}\|w\|_{Y},\ y,z,w\in Y,$$
and $A(y)\in G(X,1,\beta)$(i.e., $A(y)$ is quasi-$m$-accretive), uniformly on bounded sets in $Y$;
\item[$(ii)$] $QA(y)Q^{-1}=A(y)+B(y)$, where $B(y)\in L(X)$ is bounded, uniformly on bounded sets in $Y$, and $Q$ is an isomorphism from $Y$ onto $X$. Moreover,
$$\|(B(y)-B(z))w\|_{X}\leq \rho_{2}\|y-z\|_{Y}\|w\|_{X},\ y,z\in Y, w\in X;$$
\item[$(iii)$] $f:Y\rightarrow Y$ is bounded on any bounded sets in $Y$ and also extends to a mapping from $X$ onto $X$ and satisfies
$$\|f(y)-f(z)\|_{Y}\leq \rho_{3}\|y-z\|_{Y},\ y,z\in Y,$$
$$\|f(y)-f(z)\|_{X}\leq \rho_{4}\|y-z\|_{X},\ y,z\in X,$$
\end{description}
where $\rho_{i}(i=1,2,3)$ depend only on $\max\{\|y\|_{Y},\|z\|_{Y}\}$ and $\rho_{4}$ depends only on $\max\{\|y\|_{X},\|z\|_{X}\}$.

\begin{theorem}[Kato's theorem \cite{Kato}]
\label{the1}  Assume that $(i)-(iii)$ hold. Given $z_{0}\in Y$, there exists a maximal $T>0$ depending only on $\|z_{0}\|_{Y}$ and unique solution $z$ to equation \eqref{201} such that
$$z=z(\cdot,z_{0})\in C([0,T);Y)\cap C^{1}([0,T);X).$$
Moreover, the mapping $z_{0}\mapsto z(\cdot,z_{0})$ is continuous from $Y$ to $C([0,T);Y)\cap C^{1}([0,T);X).$
\end{theorem}

In order to apply Kato's theorem, we reformulate system \eqref{101} as follows
\begin{equation}
\begin{cases}
u_{t}+k u^{2}u_{x}=-\frac{k}{2}\Lambda^{-2}\left(u_{x}^{3}-u_{x}\rho^{2}\right)
-\partial_{x}\Lambda^{-2}\left(g(u)+\frac{3k}{2}uu_{x}^{2}-\frac{k}{3}u^{3}
-\frac{k}{2}u\rho^{2}\right)-\lambda u,\\
\rho_{t}+k u^{2}\rho_{x}=-k \rho uu_{x},
\label{202}
\end{cases}
\end{equation}
where $\Lambda^{-2}=(1-\partial_{x}^{2})^{-1}$.
Denote $z:=\left(
\begin{array}{cc}  u\\ \rho \end{array} \right)$, $A(z):=\left(
\begin{array}{cc}  k u^{2}\partial_{x}&0\\ 0&k u^{2}\partial_{x}\end{array} \right)$,
\begin{align*}
f(z):=
\left(
\begin{array}{cc}  -\frac{k}{2}\Lambda^{-2}\left(u_{x}^{3}-u_{x}\rho^{2}\right)
-\partial_{x}\Lambda^{-2}\left(g(u)+\frac{3k}{2}uu_{x}^{2}-\frac{k}{3}u^{3}
-\frac{k}{2}u\rho^{2}\right)-\lambda u\\
-k\rho uu_{x}\end{array} \right),
\label{204}
\end{align*}
$Y=H^{s}\times H^{s-1}, X=H^{s-1}\times H^{s-2}, \Lambda=(1-\partial_{x}^{2})^{\frac{1}{2}}$ and $Q=\left(
\begin{array}{cc}  \Lambda,&0\\ 0,&\Lambda\end{array} \right)$. It is obvious that $Q$ is an isomorphism of $H^{s}\times H^{s-1}$ onto $H^{s-1}\times H^{s-2}$.

In what follows, we prove the local well-posedness of solutions by using Kato's theorem.
To this end, we need the following lemmas.
\begin{lemma}[\cite{Pazy1983}]
\label{lem202}
Let $X$ and $Y$ be two Banach spaces and $Y$ be continuously and densely embedded
in $X$. Let $-A$ be the infinitesimal generator of the $C_0$-semigroup $T(t)$
on $X$ and $S$ be an isomorphism from $Y$ onto $X$. Then $Y$ is $A$-admissible
 if and only if $-A_1=-SAS^{-1}$ is the infinitesimal generator of the
 $C_0$-semigroup $T_1(t)=STS^{-1}$ on $X$. Moreover, if $Y$ is $A$-admissible,
 then the part of $-A$ in $Y$ is the infinitesimal generator of the restriction
 of $T(t)$ to $Y$.
\end{lemma}

\begin{lemma}[\cite{Kato}]
\label{lem203}  Let $r, \tau$ be real numbers such that $-r<\tau\leq r$. Then
\begin{eqnarray*}
\|fg\|_{\tau}\leq c\|f\|_{r}\|g\|_{\tau}, \ if\ r>1/2,
\\
\|fg\|_{r+\tau-\frac{1}{2}}\leq c\|f\|_{r}\|g\|_{\tau}, \ if\ r<1/2,
\end{eqnarray*}
where $c$ is a positive constant depending only on $r$ and $\tau$.
\end{lemma}

\begin{lemma}[\cite{Kato1}]
\label{lem204}  Let $f\in H^{s}, s>3/2$. Then
\begin{align*}
\|\Lambda^{-r}[\Lambda^{r+m+1},M_{f}]\Lambda^{-m}\|_{L(L^{2})}\leq c\|f\|_{H^{s}},\ |r|,|m|\leq s-1,
\end{align*}
where $M_{f}$ is the operator of multiplication by $f$ and $c$ is a constant depending only on $r$ and $m$.
\end{lemma}

\begin{lemma}[\cite{Kato1983}]
\label{lem205}  Let $h\in C^{l}(\mathbb{R},\mathbb{R})$ and $h(0)=0$. Then
$$\|h(u)\|_{r}\leq \widetilde{h}(\|u\|_{r}),\ 1/2<r\leq l.$$
Moreover, if $h\in C^{\infty}(\mathbb{R},\mathbb{R})$ with $h(0)=0$, then
$$\|h(u)\|_{r}\leq \widetilde{h}(\|u\|_{r}),\ r>1/2,$$
where $\widetilde{h}$ is a monotone increasing function depending only on the function $h$.
\end{lemma}

In what follows, we will verify that $A(z)$ and $f(z)$ satisfy the conditions (i)--(iii) in Kato's theorem.

\begin{lemma}
\label{lem206}  The operator $A(z):=\left(
\begin{array}{cc}k u^{2}\partial_{x}&0\\ 0&k u^{2}\partial_{x}\end{array} \right)$ with $z\in H^{s}\times H^{s-1} (s>\frac{3}{2})$, belongs to $G(L^{2}\times L^{2},1,\beta).$
\end{lemma}
\begin{proof}
Since $L^{2}\times L^{2}$ is a Hilbert space, we have $A(z)\in G(L^{2}\times L^{2},1,\beta)$ for some real number $\beta$ if and only if the following conditions hold \cite{Kato1983}:\\
(a) $\left(A(z)v,v\right)_{0}\geq -\beta\|v\|_{0}^{2};$\\
(b) The range of $\alpha I+A$ is all of $X$, for all $\alpha>\beta$.

First, we prove (a). Since $z:=\left(
\begin{array}{cc}  u\\ \rho \end{array} \right)\in H^{s}\times H^{s-1} (s>3/2)$, it follows that
$$u\in H^{s},\ \|u\|_{L^{\infty}}\leq \|u\|_{s}\ \text{and}\ \|u_{x}\|_{L^{\infty}}\leq \|u\|_{s}.$$
Then
\begin{align*}
\left|\left(A(z)v,v\right)_{0\times 0} \right|&=\left|(ku^{2}\partial_{x}v_{1},v_{1})_{0}+(ku^{2}\partial_{x}v_{2},v_{2})_{0}\right|\nonumber\\
&\leq |k |\left(|(uu_{x}v_{1},v_{1})_{0}|+|(uu_{x}v_{2},v_{2})_{0}|\right)\nonumber\\
&\leq |k |\|u\|_{L^{\infty}}\|u_{x}\|_{L^{\infty}}(\|v_{1}\|_{0}^{2}+\|v_{2}\|_{0}^{2})\nonumber\\
&\leq |k |\|u\|_{s}^{2}\|v\|_{0\times 0}^{2}.
\end{align*}
Setting $\beta=|k|\|u\|_{s}^{2}$, we have $\left(A(z)v,v\right)_{0}\geq -\beta\|v\|_{0\times0}^{2}.$

Second, we prove (b). Because $A(z)$ is a closed operator and satisfies (a), $(\alpha I+A)$ has closed range in $L^{2}\times L^{2}$ for all $\alpha>\beta$. Therefore, it suffices to prove that $(\alpha I+A)$ has dense range in $L^{2}\times L^{2}$ for all $\alpha>\beta$.

Given $u\in H^{s} (s>3/2)$, and $v_{1}\in L^{2}$, we have the generalized Leibnitz formula
$$\partial_{x}(u^{2}v_{1})=2uu_{x}v_{1}+u^{2}v_{1x}\ \ \text{in}\ \ H^{-1}.$$
Since $u, u_{x}\in L^{\infty}$, we obtain
\begin{align*}
D(A)=D\left(\left(
\begin{array}{cc}  k u^{2}\partial_{x},&0\\ 0,&k  u^{2}\partial_{x}\end{array} \right)\right)=&\left\{\left(
\begin{array}{cc}  v_{1}\\ v_{2} \end{array} \right)\in L^{2}\times L^{2}, \left(
\begin{array}{cc}  k u^{2}\partial_{x}v_{1}\\ k u^{2}\partial_{x}v_{2} \end{array} \right)\in L^{2}\times L^{2}\right\}\nonumber\\
=&\left\{\left(
\begin{array}{cc}  w_{1}\\ w_{2} \end{array} \right)\in L^{2}\times L^{2}, \left(
\begin{array}{cc}  -\partial_{x}(k u^{2}w_{1})\\  -\partial_{x}(k u^{2}w_{2}) \end{array} \right)\in L^{2}\times L^{2}\right\}\nonumber\\
=&D\left(\left(
\begin{array}{cc}  k  u^{2}\partial_{x},&0\\ 0,&k  u^{2}\partial_{x}\end{array} \right)^{\ast}\right)=D(A^{\ast}).
\end{align*}
Assume that the range of $(\alpha I+A)$ is not all of $L^{2}\times L^{2}$, then there exists $0\neq w\in L^{2}\times L^{2}$ such that $\left((\alpha I+A)v,w\right)_{0\times 0}=0$ for all $v\in D(A)$. Since $H^{1}\times H^{1}\subset D(A), D(A)$ is dense in $L^{2}\times L^{2}$. Therefore, it follows that $w\in D(A^{\ast})$ and $\alpha w+A^{\ast}w=0$ in $L^{2}\times L^{2}$. Since $D(A)=D(A^{\ast})$, multiplying by $w$ and integrating by parts, we get
$$0=\left((\alpha I+A^{\ast})w,w\right)_{0\times 0}=(\alpha w,w)_{0\times 0}+(w,Aw)_{0\times 0}\geq(\alpha-\beta)\|w\|_{0\times 0}^{2},\ \forall\ \alpha>\beta,$$
and thus $w=0$, which contradicts our assumption $w\neq 0$.
This proves $(b)$ and completes the proof of Lemma \ref{lem206}.
\end{proof}

\begin{lemma}
\label{lem207}  The operator $A(z):=\left(
\begin{array}{cc}  k u^{2}\partial_{x},&0\\ 0,&k  u^{2}\partial_{x}\end{array} \right)$ with $z\in H^{s}\times H^{s-1} (s>\frac{3}{2})$, belongs to $G(H^{s-1}\times H^{s-2},1,\beta).$
\end{lemma}
\begin{proof}
Since $H^{s-1}\times H^{s-2}$ is a Hilbert space, $A(z)$ belongs to $G(H^{s-1}\times H^{s-2},1,\beta)$ for some real number $\beta$ if and only if the following conditions hold \cite{Kato1983}:\\
$(a)$  $(A(z)v,v)_{(s-1)\times(s-2)} \geq-\beta\|v\|^{2}_{(s-1)\times(s-2)}$;\\
$(b)$  $-A(z)$ is the infinitesimal generator of a $C_{0}$-semigroup on $H^{s-1}\times H^{s-2}$.

First, we prove $(a)$. Due to  $u\in H^{s}, s>\frac{3}{2}, u$ and $u_{x}$ belong to  $L^{\infty}$, and $\|u\|_{L^{\infty}}, \|u_{x}\|_{L^{\infty}}\leq\|u\|_{s}$. Note that
\begin{align*}
\Lambda^{s-1}\left(k u^{2} \partial_{x} v_{1}\right)
&=\left[\Lambda^{s-1}, k u^{2}\right] \partial_{x} v_{1}+k u^{2} \Lambda^{s-1}\left(\partial_{x} v_{1}\right) \\
&=\left[\Lambda^{s-1}, k u^{2}\right] \partial_{x} v_{1}+k u^{2} \partial_{x} \Lambda^{s-1} v_{1},
\end{align*}
and
$$\Lambda^{s-2}\left(k u^{2} \partial_{x} v_{2}\right)=\left[\Lambda^{s-2}, k u^{2}\right] \partial_{x} v_{2}+k u^{2} \partial_{x}\Lambda^{s-2} v_{2} .$$
Then we have
\begin{align*}
&(A(z)v,v)_{(s-1)\times(s-2)}\nonumber\\
=&\left(\Lambda^{s-1}\left(k u^{2} \partial_{x}v_{1}\right), \Lambda^{s-1} v_{1}\right)_{0}+\left(\Lambda^{s-2}\left(k u^{2}\partial_{x} v_{2}\right), \Lambda^{s-2} v_{2}\right)_{0}\nonumber\\
=&\left(\left[\Lambda^{s-1},k u^{2}\right] \partial_{x}v_{1}, \Lambda^{s-1} v_{1}\right)_{0}-\left(k uu_{x}\Lambda^{s-1}v_{1}, \Lambda^{s-1}v_{1}\right)_{0}\nonumber\\
&+\left(\left[\Lambda^{s-2},k u^{2}\right] \partial_{x}v_{2}, \Lambda^{s-2} v_{2}\right)_{0}-\left(k uu_{x} \Lambda^{s-2}v_{2}, \Lambda^{s-2}v_{2}\right)_{0} \nonumber\\
\leq&\left\|\left[\Lambda^{s-1}, k u^{2}\right] \Lambda^{2-s}\right\|_{L(L^{2})}\left\|\Lambda^{s-1} v_{1}\right\|_{0}^{2}+|k|\|u\|_{L^{\infty}}\|u_{x}\|_{L^{\infty}}\left\|\Lambda^{s-1} v_{1}\right\|_{0}^{2} \nonumber\\
&+\left\|\left[\Lambda^{s-2}, k u^{2}\right] \Lambda^{3-s}\right\|_{L(L^{2})}\left\|\Lambda^{s-2} v_{2}\right\|_{0}^{2}+|k|\|u\|_{L^{\infty}}\|u_{x}\|_{L^{\infty}}\left\|\Lambda^{s-2} v_{2}\right\|_{0}^{2}\nonumber\\
\leq&\left(C\|u\|_{s}+|k|\|u\|_{s}^{2}\right)\left(\|v_{1}\|_{s-1}^{2}+\|v_{2}\|_{s-2}^{2}\right) \nonumber\\
=&\left(C\|u\|_{s}+|k|\|u\|_{s}^{2}\right)\|v\|^{2}_{(s-1)\times(s-2)}\nonumber\\
\leq&\left(C\|z\|_{s\times(s-1)}+|k|\|z\|^{2}_{s\times(s-1)}\right)\|v\|^{2}_{(s-1) \times(s-2)},
\end{align*}
where we applied Lemma \ref{lem204} with $r=0, k=s-2$ and  with $r=0, k=s-3$.
Let $\beta=C\|z\|_{s\times(s-1)}+|k|\|z\|^{2}_{s\times(s-1)}$, we have $$\left(A(z)v,v\right)_{(s-1)\times(s-2)}\geq -\beta\|v\|^{2}_{(s-1) \times(s-2)}.$$

Second, we prove (b). Let $Q=\left(
\begin{array}{cc}  \Lambda^{s-1},&0\\ 0,&\Lambda^{s-2}\end{array} \right)$. Note that $Q$ is an isomorphism of $H^{s-1}\times H^{s-2}$ onto $L^{2}\times L^{2}$ and $H^{s-1}\times H^{s-2}$ is continuously and densely embeded in $L^{2}\times L^{2}$ as $s>3/2$.
Define
$$A_{1}(z)=QA(z)Q^{-1}=\Lambda^{s-1}A(z)\Lambda^{1-s},\ \ B_{1}(z)=A_{1}(z)-A(z).$$

Note that
\begin{equation*}
B_{1}(z)v=\left[\Lambda^{s-1},k u^{2}\partial_{x}\right] \Lambda^{1-s}v_{1}+\left[\Lambda^{s-2},k u^{2} \partial_{x}\right] \Lambda^{2-s} v_{2}.
\end{equation*}
Let $v\in L^{2}\times L^{2}$ and $u\in H^{s} (s>3/2)$. Then we have
\begin{align*}
&\left\|B_{1}(z)v\right\|_{0 \times 0} \\
=&\left\|\left[\Lambda^{s-1}, ku^{2}\partial_{x}\right] \Lambda^{1-s} v_{1}\right\|_{0}+\left\|\left[\Lambda^{s-2}, ku^{2}\partial_{x}\right] \Lambda^{2-s} v_{2}\right\|_{0} \\
\leq&\left\|\left[\Lambda^{s-1}, k u^{2}\right] \Lambda^{2-s}\right\|_{L\left(L^{2}\right)}\left\|\Lambda^{-1} \partial_{x} v_{1}\right\|_{0}+\left\|\left[\Lambda^{s-2}, k u^{2}\right] \Lambda^{3-s}\right\|_{L\left(L^{2}\right)}\left\|\Lambda^{-1} \partial_{x} v_{2}\right\|_{0} \\
\leq&C\|u\|_{s}\|v\|_{L^{2} \times L^{2}} \\
\leq& C\|z\|_{s\times(s-1)}\|v\|_{L^{2} \times L^{2}},
\end{align*}
where we used the Lemma \ref{lem203} with $r=0, t=s-2$ and with $r=0, t=s-3$.
Therefore, we obtain $B_{1}(z)\in L(L^{2}\times L^{2})$.

Note that $A_{1}(z)=B_{1}(z)+A(z)$ and $A(z)\in G(L^{2}\times L^{2},1,\beta)$, by a perturbation theorem of semigroups \big(\cite{Pazy1983}, Section 5.2, Theorem 2.3\big), we obtain $A_{1}(z)\in G(L^{2}\times L^{2},1,\beta_{1})$.

By applying Lemma \ref{lem202} with $X=L^{2}\times L^{2}, Y=H^{s-1}\times H^{s-2},\ Q=\left(
\begin{array}{cc}  \Lambda^{s-1}&0\\ 0&\Lambda^{s-2}\end{array} \right)$, we can conclude that $H^{s-1}\times H^{s-2}$ is $A$-admissible. Therefore, $-A(u)$ is the infinitesimal generator of a $C_{0}$-semigroup on $H^{s-1}\times H^{s-2}.$
This proves $(b)$ and completes the proof of Lemma \ref{lem207}.
\end{proof}

\begin{lemma}
\label{lem208}  Let operator $A(z):=\left(
\begin{array}{cc}  k u^{2}\partial_{x}&0\\ 0&k  u^{2}\partial_{x}\end{array} \right)$ with $z\in H^{s}\times H^{s-1} (s>\frac{3}{2})$ be given. Then operator $A(z)\in L(H^{s}\times H^{s-1},H^{s-1}\times H^{s-2})$ and satisfies
$$\|(A(z)-A(v))w\|_{(s-1)\times(s-2)}\leq \mu_{1}\|z-v\|_{(s-1)\times(s-2)}\|w\|_{s\times(s-1)}$$
for all $z, v, w\in H^{s}\times H^{s-1}.$
\end{lemma}

\begin{proof}
Let $z, v, w\in H^{s}\times H^{s-1}, s>\frac{3}{2}$. Note that $H^{s-1}$ is a Banach algebra and
\begin{equation*}
(A(z)-A(v))w=\left(\begin{array}{cc}
k\left(u^{2}-v_{1}^{2}\right) \partial_{x} & 0 \\
0 & k\left(u^{2}-v_{1}^{2}\right) \partial_{x}
\end{array}\right)\left(\begin{array}{l}
w_{1} \\
w_{2}
\end{array}\right)=\left(\begin{array}{c}
k\left(u^{2}-v_{1}^{2}\right) \partial_{x}w_{1} \\
k\left(u^{2}-v_{1}^{2}\right) \partial_{x}w_{2}
\end{array}\right).
\end{equation*}
Then we have
\begin{align*}
&\|(A(z)-A(v)) w\|_{(s-1)\times(s-2)}\nonumber\\
\leq&|k|\left\|\left(u^{2}-v_{1}^{2}\right) \partial_{x} w_{1}\right\|_{s-1}+|k|\left\|\left(u^{2}-v_{1}^{2}\right) \partial_{x} w_{2}\right\|_{s-2} \nonumber\\
\leq&|k|\left\|u^{2}-v_{1}^{2}\right\|_{s-1}\left\|\partial_{x} w_{1}\right\|_{s-1}+|k|\left\|u^{2}-v_{1}^{2}\right\|_{s-2}\left\|\partial_{x} w_{2}\right\|_{s-2} \nonumber\\
=&|k|\left\|\left(u+v_{1}\right)\left(u-v_{1}\right)\right\|_{s-1}
\left\|w_{1}\right\|_{s}+|k|\left\|\left(u+v_{1}\right)
\left(u-v_{1}\right)\right\|_{s-2}\left\|w_{2}\right\|_{s-1} \nonumber\\
\leq&|k|\left(\|u\|_{s-1}+\|v\|_{s-1}\right)
\left\|u-v_{1}\right\|_{s-1}\left\|w_{1}\right\|_{s}+| k|\left(\|u\|_{s-2}+\|v\|_{s-2}\right)\left\|u-v_{1}\right\|_{s-2}
\left\|w_{2}\right\|_{s-1} \nonumber\\
\leq&\mu_{1}\|z-v\|_{(s-1)\times(s-2)}\|w\|_{s\times(s-1)},
\end{align*}
where $\mu_{1}=\mu_{1}(\max\{|k|,\|z\|_{s\times(s-1)},\|v\|_{s\times(s-1)}\})$. Here we used Lemma \ref{lem203} with $r=s-1, t=s-2$. Taking $v=0$ in the above inequality, we obtain that $A(z)\in L(H^{s}\times H^{s-1},H^{s-1}\times H^{s-2})$. This completes the proof of Lemma \ref{lem208}.
\end{proof}

\begin{lemma}
\label{lem209}  Let $B(z)=QA(z)Q^{-1}-A(z)$ with $z\in H^{s}\times H^{s-1} (s>\frac{3}{2})$. Then $B(z)\in L(H^{s-1}\times H^{s-2})$ and
$$\|(B(z)-B(v))w\|_{(s-1)\times (s-2)}\leq \mu_{2}\|z-v\|_{s\times (s-1)}\|w\|_{(s-1)\times (s-2)},$$
holds for all $z,v\in H^{s}\times H^{s-1}$ and $w\in H^{s-1}\times H^{s-2}.$
\end{lemma}
\begin{proof}
Let $z, v \in H^{s}\times H^{s-1}$ and $w\in H^{s-1}\times H^{s-2}$. Note that
$$(B(z)-B(v))w=\left[\Lambda,k\left(u^{2}-v_{1}\right)^{2}\partial_{x}\right] \Lambda^{-1} w_{1}+\left[\Lambda, k\left(u^{2}-v_{1}\right)^{2} \partial_{x}\right] \Lambda^{-1}w_{2}.$$
Then we have
\begin{align*}
&\|(B(z)-B(v))w\|_{(s-1)\times(s-2)}\nonumber\\
\leq&\left\|\Lambda^{s-1}\left[\Lambda,k\left(u^{2}-v_{1}^{2}\right)\partial_{x}\right]\Lambda^{-1}w_{1}\right\|_{L^{2}}+\left\|\Lambda^{s-2}\left[\Lambda, k\left(u^{2}-v_{1}^{2}\right)\partial_{x}\right]\Lambda^{-1} w_{2}\right\|_{L^{2}}\nonumber\\
\leq&\left\|\Lambda^{s-1}\left[\Lambda,k\left(u^{2}-v_{1}^{2}\right)\right] \Lambda^{1-s}\right\|_{L(L^{2})}\left\|\Lambda^{s-2} \partial_{x} w_{1}\right\|_{L^{2}}\nonumber\\
&+\|\Lambda^{s-2}\left[\Lambda,k\left(u^{2}-v_{1}^{2}\right) \Lambda^{2-s}\right\|_{L(L^{2})}\|\Lambda^{s-3}\partial_{x}w_{2}\|_{L^{2}}\nonumber\\
\leq&\mu_{2}\|z-v\|_{s\times(s-1)}\|w\|_{(s-1)\times(s-2)},
\end{align*}
where $\mu_{2}=\mu_{2}(\max\{|k|,\|z\|_{s\times(s-1)},\|v\|_{s\times(s-1)}\})$. Here we applied Lemma \ref{lem203} with  $r=1-s, k=s-1$ and $r=2-s, k=s-2$. Taking $v=0$ in the above inequality, we obtain $B(z)\in L(H^{s-1}\times H^{s-2})$. This completes the proof of Lemma \ref{lem209}.
\end{proof}

\begin{lemma}
\label{lem2010}  Let $z\in H^{s}\times H^{s-1}, s>\frac{3}{2},$ and
\begin{align*}
f(z):=&
\left(
\begin{array}{cc}  -\frac{k}{2}\Lambda^{-2}\left(u_{x}^{3}-u_{x} \rho^{2}\right)-\partial x \Lambda^{-2}\left(g_{(u)}+\frac{3k}{2} uu_{x}^{2}-\frac{k}{3} u^{3}-\frac{k}{2}u\rho^{2}\right)-\lambda u\\
-k \rho uu_{x}\end{array} \right).
\end{align*}
Then $f(z)$ is uniformly bounded on any bounded sets in $H^{s}\times H^{s-1}$, and satisfies the following two conditions:\\
(a) $\|f(z)-f(v)\|_{s\times(s-1)}\leq \mu_{3}\|z-v\|_{s\times(s-1)},\ z, v\in H^{s}\times H^{s-1};$\\
(b) $\|f(z)-f(v)\|_{(s-1)\times(s-2)}\leq \mu_{4}\|z-v\|_{(s-1)\times(s-2)},\ z, v\in H^{s-1}\times H^{s-2}.$
\end{lemma}
\begin{proof}
First, we prove (a). Let $z, v\in H^{s}\times H^{s-1} (s>\frac{3}{2})$. Note that $H^{s-1}$ is a Banach algebra, then we have
\begin{align}
&\|f(z)-f(v)\|_{s\times(s-1)}\nonumber\\
=&\Big\|-\frac{k}{2}\Lambda^{-2}\left[\left(u_{x}^{3}-v_{1x}^{3}\right)
-\left(u_{x}\rho^{2}-v_{1x}v_{2}^{2}\right)\right]
-\partial_{x}\Lambda^{-2}[(g(u)-g(v))+\frac{3k}{2}\left(uu_{x}^{2}
-v_{1}v_{1x}^{2}\right)\nonumber\\
&-\frac{k}{3}\left(u^{3}-v_{1}^{3}\right)+\frac{k}{2}\left(u\rho^{2}-v_{1} v_{2}^{2}\right)-\lambda(u-v_{1})\Big\|_{s}+|k|\|\rho uu_{x}-v_{1}v_{1x}v_{2}\|_{s-1}\nonumber\\
\leq& \frac{|k|}{2}\left\|u_{x}^{3}-v_{1x}^{3}\right\|_{s-2}+\frac{|k|}{2}\left\|u_{x} \rho^{2}-v_{1x}v_{2}^{2}\right\|_{s-2}+\|g(u)-g(v)\|_{s-1}
+\frac{3|k|}{2}\left\|uu_{x}^{2}-v_{1}v_{1x}^{2}\right\|_{s-1}\nonumber\\
&+\frac{|k|}{3}\left\|u^{3}-v_{1}^{3}\right\|_{s-1}
+\frac{|k|}{2}\left\|u\rho^{2}-v_{1}v_{2}^{2}\right\|_{s-1}
+\lambda\left\|u-v_{1}\right\|_{s-1}+|k|\left\|\rho uu_{x}-v_{1}v_{1x}v_{2}\right\|_{s-1}.
\label{203}
\end{align}

In what follows, we will estimate the terms for the right side of \eqref{203}.
Note that
\begin{align*}
\left\|u_{x}^{3}-v_{1x}^{3}\right\|_{s-2}
\leq& C\left\|u_{x}-v_{1x}\right\|_{s-2}
\left\|u_{x}^{2}+u_{x}v_{1x}+v_{1x}^{2}\right\|_{s-2}\nonumber\\
\leq& C\left(\|u\|_{s}^{2}+\|v_{1}\|_{s}^{2}\right)\left\|u-v_{1}\right\|_{s}\nonumber\\
\leq& C\left(\|z\|_{s\times(s-1)}^{2}+\|v\|_{s\times(s-1)}^{2}\right)
\left\|z-v\right\|_{s\times(s-1)},
\end{align*}
\begin{align*}
\left\|u_{x}\rho^{2}-v_{1x}v_{2}^{2}\right\|_{s-2}
\leq&\left\|u_{x}\rho^{2}-u_{x}v_{2}^{2}\right\|_{s-2}+\left\|u_{x}v_{2}^{2}-v_{1x}v_{2}^{2}\right\|_{s-2} \nonumber\\
\leq&C\|u_{x}\|_{s-2}\left\|\rho^{2}-v_{2}^{2}\right\|_{s-2}+C\left\|u_{x}-v_{1x}\right\|_{s-2}
\left\|v_{2}^{2}\right\|_{s-2}\nonumber\\
\leq& C\|u\|_{s-1}\left\|\left(\rho+v_{2}\right)\left(\rho-v_{2}\right)\right\|_{s-2}
+C\left\|u-v_{1}\right\|_{s-1}\left\|v_{2}^{2}\right\|_{s-2}\nonumber\\
\leq& C\|u\|_{s-1}\left(\|\rho\|_{s-2}+\left\|v_{2}\right\|_{s-2}\right)
\left\|\rho-v_{2}\right\|_{s-2}+C\|v\|_{s}^{2}\left\|u-v_{1}\right\|_{s-1} \nonumber\\
\leq& C\|z\|_{s\times(s-1)}\left(\|z\|_{s\times(s-1)}+\|v\|_{s\times(s-1)}
\left\|\rho-v_{2}\right\|_{s-1}\right)+C\|v\|_{s}^{2}\left\|u-v_{1}\right\|_{s} \nonumber\\
\leq& C\left(\|z\|_{s\times(s-1)}^{2}+\|v\|_{s\times(s-1)}^{2}\right)\|z-v\|_{s\times(s-1)}.
\end{align*}
Similarly, we can estimate
\begin{eqnarray*}
\left\|u^{3}-v_{1}^{3}\right\|_{s-1}\leq C\left(\|z\|_{s\times(s-1)}^{2}+\|v\|_{s\times(s-1)}^{2}\right)\left\|z-v\right\|_{s\times(s-1)},\\
\label{207}
\left\|uu_{x}^{2}-v_{1}v_{1x}^{2}\right\|_{s-1}
\leq C\left(\|z\|_{s\times(s-1)}^{2}+\|v\|_{s\times(s-1)}^{2}\right)\left\|z-v\right\|_{s\times(s-1)},\\
\left\|u\rho^{2}-v_{1}v_{2}^{2}\right\|_{s-1}
\leq C\left(\|z\|_{s\times(s-1)}^{2}+\|v\|_{s\times(s-1)}^{2}\right)\left\|z-v\right\|_{s\times(s-1)}.
\label{208}
\end{eqnarray*}
Meanwhile, from Lemma \ref{lem206}, we know that $u\mapsto g(u)$ is a $C^{1}$-map from $H^{s-1}$ to $H^{s-1}$. By using the mean value theorem, we can infer that there is some $M>0$, depending only on $\max\{\|u\|_{s},\|v_{1}\|_{s}\}$ such that
$$\|g(u)-g(v_{1})\|_{s-1}\leq M\|u-v_{1}\|_{s-1}\leq M\|u-v_{1}\|_{s}
\leq M\|z-v\|_{s\times(s-1)}.$$
Therfore, we get
\begin{equation*}
\|f(z)-f(v)\|_{s\times(s-1)}\leq \mu_{3}\|z-v\|_{s\times(s-1)},
\end{equation*}
where $\mu_{3}=\mu_{3}(\max\{\|z\|_{s\times(s-1)}, \|v\|_{s\times(s-1)}\})$.
Take $v=0$ in the above inequality, we obtain that $f(z)$ is bounded on bound set in $H^{s}\times H^{s-1}$.

Second, similar to the above process, we can get
\begin{equation*}
\|f(z)-f(v)\|_{(s-1)\times(s-2)}\leq\mu_{4}\|z-v\|_{(s-1)\times(s-2)},
\end{equation*}
where $\mu_{4}=\mu_{4}(\max\{\|z\|_{(s-1)\times(s-2)}, \|v\|_{(s-1)\times(s-2)}\})$.
This completes the proof of Lemma \ref{lem2010}.
\end{proof}

Combining Lemmas \ref{lem206}--\ref{lem2010} and applying Kato's theorem, we can obtain the following local well-posedness result.
\begin{theorem}[Local well-posedness]
\label{the201} Given $z_{0}:=\left(
\begin{array}{cc}  u_{0}\\ \rho_{0}\end{array} \right)\in H^{s}\times H^{s-1} (s>3/2)$, then there exists a maximal time $T=T(\lambda, z_{0})>0$ and a unique solution $z:=\left(
\begin{array}{cc}  u\\ \rho\end{array} \right)$ to \eqref{101} such that
$$z=z(\cdot,z_{0})\in C([0,T);H^{s}\times H^{s-1})\cap C^{1}([0,T);H^{s-1}\times H^{s-2}).$$
Moreover, the solution $z$ depends continuously on the initial data, i.e., the map
$$z_{0}\mapsto z(\cdot,z_{0}):H^{s}\times H^{s}\rightarrow C([0,T);H^{s}\times H^{s})\cap C^{1}([0,T);H^{s-1}\times H^{s-1})$$
is continuous.
\end{theorem}

\section{Wave breaking}
In this section, we will derive the necessary and sufficient condition for occurrence of wave breaking.
In what follows, we give a conserved quantity of solutions $z$ for \eqref{101} with the initial data $z_{0}(x)$.
\begin{lemma}
\label{lem301} Let $z_{0}:=\left(
\begin{array}{cc}  u_{0}\\ \rho_{0}\end{array} \right)\in H^{s}\times H^{s-1} (s>3/2)$, and $T=T(\lambda, z_{0})>0$ be the maximal existence time of the corresponding solution $z:=\left(\begin{array}{cc}  u\\ \rho\end{array} \right)$ to \eqref{101}. Then for all $t\in [0,T)$, we have
\begin{equation*}
\int_{\mathbb{R}}(u^{2}+u_{x}^{2})(t,x)dx=e^{-2\lambda t}\int_{\mathbb{R}}(u_{0}^{2}+u_{0x}^{2})dx,\ \
\int_{\mathbb{R}}\rho^{2}(t,x)dx=\int_{\mathbb{R}}\rho_{0}^{2}dx.
\end{equation*}
\end{lemma}
\begin{proof}
Multiplying the first equation in \eqref{101} by $2u$, we get
\begin{equation}
2uu_{t}-2uu_{txx}+2u\left(g(u)\right)_{x}+2\lambda(u^{2}-uu_{xx})
=2k\left(3u^{2}u_{x}u_{xx}+u^{3}u_{xxx}\right)+2k u\rho(u\rho)_{x}.
\label{301}
\end{equation}
Integrating \eqref{301} with respect to $x$ over $\mathbb{R}$, we get
\begin{equation}
\frac{d}{dt}\int_{\mathbb{R}}(u^{2}+u_{x}^{2})dx+2\lambda \int_{\mathbb{R}}(u^{2}+u_{x}^{2})dx=0,
\label{302}
\end{equation}
here we used the relations that
\begin{align*}
2\int_{\mathbb{R}}uu_{txx}dx&=-2\int_{\mathbb{R}}u_{x}u_{tx}dx,\
2\int_{\mathbb{R}}uu_{xx}dx=-2\int_{\mathbb{R}}u^{2}_{x}dx,\\
\int_{\mathbb{R}}u^{3}u_{xxx}dx&=-3\int_{\mathbb{R}}u^{2}u_{x}u_{xx}dx,\
2\int_{\mathbb{R}}u(g(u))_{x}dx=-2\int_{\mathbb{R}}G'(u(x))dx,
\end{align*}
where $G(u)=\int_{0}^{u}g(v)dv$.

Integrating \eqref{302} with respect to $t$ over $[0,t]$, we get
\begin{equation*}
\int_{\mathbb{R}}(u^{2}+u_{x}^{2})(t,x)dx=e^{-2\lambda t}\int_{\mathbb{R}}(u_{0}^{2}+u_{0x}^{2})dx.
\end{equation*}

On the other hand, multiplying the second equation in \eqref{101} by $2\rho$, we get
\begin{equation*}
2\rho\rho_{t}=-2k\rho\rho_{x} u^{2}-2k \rho^{2}uu_{x}.
\end{equation*}
Integrating it with respect to $t$ and $x$ over $[0,t]\times \mathbb{R}$, we get
\begin{equation*}
\int_{\mathbb{R}}\rho^{2}(t,x)dx=\int_{\mathbb{R}}\rho_{0}^{2}dx.
\end{equation*}
This completes the proof of Lemma \ref{lem301}.
\end{proof}

In what following, we will derive the necessary and sufficient condition for occurrence of wave breaking. To this end, we can rewrite system \eqref{202} as follows
\begin{equation}
\begin{cases}
y_{t}+k u^{2}y_{x}+3k uu_{x}y=-\left(g(u)-\frac{4k}{3}u^{3}\right)_{x}-\lambda y+k \rho(u\rho)_{x},\\
\rho_{t}+k u^{2}\rho_{x}+k \rho uu_{x}=0,
\label{303}
\end{cases}
\end{equation}
where $y=u-u_{xx}$ denotes the momentum density.

In what follows, we present the some lemmas which are crucial in the proof of wave breaking mechanism.
\begin{lemma}[\cite{Kato1988}]
\label{lem302}
If  $r>0$, then  $H^{r}(\mathbb{R})\cap L^{\infty}(\mathbb{R})$ is an algebra. Moreover
$$\|fg\|_{r}\leq c\left(\|f\|_{L^{\infty}(\mathbb{R})}\|g\|_{r}+\|f\|_{r}\|g\|_{L^{\infty}(\mathbb{R})}\right),$$
where $c$ is a constant depending only on $r$.
\end{lemma}

\begin{lemma}[\cite{Kato1988}]
\label{lem303}
If $r>0$, then
$$\left\|\left[\Lambda^{r}, f\right] g\right\|_{L^{2}(\mathbb{R})} \leq c\left(\left\|\partial_{x} f\right\|_{L^{\infty}(\mathbb{R})}\left\|\Lambda^{r-1} g\right\|_{L^{2}(\mathbb{R})}+\left\|\Lambda^{r} f\right\|_{L^{2}(\mathbb{R})}\|g\|_{L^{\infty}(\mathbb{R})}\right),$$
where $c$ is a constant depending only on $r$.
\end{lemma}

\begin{theorem}
\label{the301}
Let $z_{0}:=\left(
\begin{array}{cc} u_{0}\\ \rho_{0}\end{array}\right)\in H^{s}\times H^{s-1} (s>3/2)$ be given. And assume that $T$ is the maximal existence time of the corresponding solution $z=\left(\begin{array}{cc}u\\ \rho\end{array}\right)$ to \eqref{202} with the initial data $z_{0}$. If there exists $M>0$, such that
\begin{equation*}
\|u_{x}(t,\cdot)\|_{L^{\infty}}+\|\rho(t,\cdot)\|_{L^{\infty}}
+\|\rho_{x}(t,\cdot)\|_{L^{\infty}}\leq M,\ \ t\in[0,T),
\end{equation*}
then the $H^{s}\times H^{s-1}$-norm of $z(t,\cdot)$ does not blow up on $[0,T)$.
\end{theorem}
\begin{proof}
Let $z$ be the unique solution of \eqref{202} with initial data $z_{0}\in H^{s}\times H^{s-1} (s>3/2)$, and $T$ be the maximal existence time of the corresponding solution $z$, which is guaranteed by Theorem \ref{the201}.

Applying the operator $\Lambda^{s}$ to the first equation in \eqref{202}, multiplying by $2\Lambda^{s}u$ and integrating the resulting equation over $\mathbb{R}$, we have
\begin{equation}
\frac{d}{dt}(u,u)_{s}=-2k\left(u^{2}u_{x}, u\right)_{s}+2\left(f_{1}(u),u\right)_{s}+2(f_{2}(u,\rho),u)_{s},
\label{304}
\end{equation}
where
\begin{align*}
f_{1}(u)&=-\frac{k}{2} \Lambda^{-2}\left(u_{x}^{3}\right)-\partial_{x} \Lambda^{-2}\left(g(u)+\frac{3k}{2}uu_{x}^{2}-\frac{k}{3} u^{3}\right)-\lambda u,\\
f_{2}(u,\rho)&=\frac{k}{2} \Lambda^{-2}\left(u_{x} \rho^{2}\right)+\frac{k}{2} \partial_{x} \Lambda^{-2}\left(u \rho^{2}\right).
\end{align*}

In what follows, we estimate the right side of \eqref{304}. If there exists a positive constant $M>0$ such that
\begin{equation*}
\|u_{x}(t,\cdot)\|_{L^{\infty}}+\|\rho(t,\cdot)\|_{L^{\infty}}
+\|\rho_{x}(t,\cdot)\|_{L^{\infty}}\leq M,\ \ t\in[0,T),
\end{equation*}
then we have
\begin{align}
|(u^{2}u_{x},u)_{s}|=&|(\Lambda^{s}(u^{2}u_{x}),\Lambda^{s}u)_{0}|\nonumber\\
=&\left|\left([\Lambda^{s},u^{2}]u_{x},\Lambda^{s}u\right)_{0}
+\left(u^{2}\Lambda^{s}u_{x},\Lambda^{s}u\right)_{0}\right|\nonumber\\
\leq&\left\|[\Lambda^{s},u^{2}]u_{x}\right\|_{L^{2}}\|\Lambda^{s}u\|_{L^{2}}
+\left\|uu_{x}\right\|_{L^{\infty}}\|\Lambda^{s}u\|^{2}_{L^{2}}\nonumber\\
\leq& C\left(\|(u^{2})_{x}\|_{L^{\infty}}\|\Lambda^{s-1}u_{x}\|_{L^{2}}
+\|\Lambda^{s}u^{2}\|_{L^{2}}\|u_{x}\|_{L^{\infty}}\right)\|u\|_{s}+cM\|u\|_{s}^{2}\nonumber\\
\leq& C\left(\|uu_{x}\|_{L^{\infty}}\|u\|_{s}
+M\|u^{2}\|_{s}\right)\|u\|_{s}+cM\|u\|_{s}^{2}\nonumber\\
\leq& C\|u\|_{s}^{2}.
\label{305}
\end{align}
Similarly, note that $H^{s}(s>\frac{1}{2})$ is a Banach algebra, it follows that
\begin{align}
(f_{1}(u),u)_{s}
=&\left(-\frac{k}{2}\Lambda^{-2}\left(u_{x}^{3}\right)
-\partial_{x}\Lambda^{-2}\left(g(u)+\frac{3k}{2}uu_{x}^{2}-\frac{k}{3}u^{3}\right)-\lambda u,u\right)_{s}\nonumber\\
\leq& C\|u\|_{s}\left(\left\|\Lambda^{-2}\left(u_{x}^{3}\right)\right\|_{s}+
\left\|\partial_{x}\Lambda^{-2}\left(g(u)+\frac{3k}{2}uu_{x}^{2}
-\frac{k}{3}u^{3}\right)\right\|_{s}
+\lambda\|u\|_{s}
\right)\nonumber\\
\leq& C\|u\|_{s}\left(\|u_{x}^{3}\|_{s-2}
+\left\|g(u)\right\|_{s-1}+\left\|uu_{x}^{2}\right\|_{s-1}
+\left\|u^{3}\right\|_{s-1}+\lambda\|u\|_{s}
\right)\nonumber\\
\leq& C\|u\|_{s}\left(\|u_{x}^{3}\|_{s-2}
+\widetilde{g}(\|u\|_{L^{\infty}})\left\|u\right\|_{s-1}
+\|u\|_{L^{\infty}}\|u_{x}^{2}\|_{s-1}+\left\|u^{3}\right\|_{s-1}+\lambda\|u\|_{s}
\right)\nonumber\\
\leq& C\|u\|_{s}\left(\|u_{x}^{3}\|_{s-2}
+\widetilde{g}(\|u_{0}\|_{1})\left\|u\right\|_{s}
+\|u\|_{L^{\infty}}\|u_{x}^{2}\|_{s-1}+\left\|u^{3}\right\|_{s-1}+\lambda\|u\|_{s}
\right)\nonumber\\
\leq&(C+\lambda)\|u\|_{s}^{2},
\label{306}
\end{align}
and
\begin{align}
|(f_{2}(u,\rho),u)_{s}|
\leq&\|f_{2}(u,\rho)\|_{s}\|u\|_{s}\nonumber\\
\leq&C(\|\Lambda^{-2}(u_{x}\rho^{2})\|_{s}
+\|\partial_{x}\Lambda^{-2}(u\rho^{2})\|_{s})\|u\|_{s}\nonumber\\
=&C(\|u_{x}\rho^{2}\|_{s-2}\|u\|_{s}+\|u\rho^{2}\|_{s-1}\|u\|_{s}) \nonumber\\
\leq&C(\|u_{x}\|_{L^{\infty}}\|\rho\|_{L^{\infty}}\|\rho\|_{s-1}\|u\|_{s}
+\|u\|_{L^{\infty}}\|\rho\|_{s-1}\|\rho\|_{L^{\infty}}\|u\|_{s}) \nonumber\\
\leq&(CM^{2}+M\|u_{0}\|_{1})(\|\rho\|_{s-1}^{2}+\|u\|_{s}^{2})\nonumber\\
\leq& C(\|u\|_{s}^{2}+\|\rho\|_{s-1}^{2}),
\label{307}
\end{align}
where we used the fact that
\begin{align*}
\|uu_{x}^{2}\|_{s-1}
=&\|[\Lambda^{s-1},u]u_{x}^{2}+u\Lambda^{s-1}u_{x}^{2}\|_{L^{2}}\nonumber\\
\leq& C\left(\|u_{x}\|_{L^{\infty}}\|\Lambda^{s-1}u_{x}^{2}\|_{L^{2}}
+\|\Lambda^{s-1}u^{2}\|_{L^{2}}\|u_{x}^{2}\|_{L^{\infty}}\right)\nonumber\\
\leq&C(M\|u_{x}^{2}\|_{s-1}+M^{2}\|u^{2}\|_{s-1})\nonumber\\
\leq&C(\|u_{x}\|_{s-1}\|u_{x}\|_{L^{\infty}}+\|u\|_{s-1}\|u\|_{L^{\infty}})\nonumber\\
\leq&C\|u\|_{s}.
\end{align*}
Combining \eqref{304}--\eqref{307}, we get
\begin{equation}
\frac{d}{dt}\|u\|_{s}^{2}\leq (C+\lambda)(\|u\|_{s}^{2}+\|\rho\|_{s-1}^{2}).
\label{308}
\end{equation}

In order to derive a similar estimate for the second component $\rho$, we apply the operator $\Lambda^{s-1}$ to the second equation in \eqref{202}, multiply by $2\Lambda^{s-1}\rho$  and integrate the resulting equation over $\mathbb{R}$, we get
\begin{equation}
\frac{d}{dt}\|\rho\|_{s-1}^{2}=-2k(u^{2}\rho_{x},\rho)_{s-1}-2k(\rho uu_{x}, \rho)_{s-1}.
\label{309}
\end{equation}
In what follows, we estimate the right side of \eqref{309}.
\begin{align}
(u^{2}\rho_{x},\rho)_{s-1}&=|(\Lambda^{s-1}(u^{2}\rho_{x}),\Lambda^{s-1}\rho)_{0}| \nonumber\\
&=\left|\left([\Lambda^{s-1},u^{2}]\rho_{x},\Lambda^{s-1}\rho\right)_{0}
+\left(u^{2}\Lambda^{s-1}\rho_{x},\Lambda^{s-1}\rho\right)_{0}\right|\nonumber\\
&\leq\|[\Lambda^{s-1},u^{2}]\rho_{x}\|_{L^{2}}\|\Lambda^{s-1}\rho\|_{L^{2}}
+\|uu_{x}\|_{L^{\infty}}\|\Lambda^{s-1} \rho\|_{L^{2}}^{2}\nonumber\\
&\leq C(\|uu_{x}\|_{L^{\infty}}\|\rho\|_{s-1}
+\|\rho_{x}\|_{L^{\infty}}\|u\|_{L^{\infty}}\|u\|_{s-1})\|\rho\|_{s-1}
+M\|u_{0}\|_{1}\|\rho\|_{s-1}^{2}\nonumber\\
&<C(\|u\|_{s}^{2}+\|\rho\|_{s-1}^{2}),
\label{3010}
\end{align}
\begin{equation}
\|(\rho uu_{x},\rho)\|_{s-1}\leq\|\rho uu_{x}\|_{s-1}\|\rho\|_{s-1}
\leq\|uu_{x}\|_{L^{\infty}}\|\rho\|_{s-1}^{2}\leq C\|\rho\|_{s-1}^{2} \\
\leq C(\|u\|_{s}^{2}+\|\rho\|_{s-1}^{2}).
\label{3011}
\end{equation}
Combining \eqref{309}--\eqref{3011}, we get
\begin{equation}
\frac{d}{dt}\|\rho\|_{s-1}^{2}\leq C(\|u\|_{s}^{2}+\|\rho\|_{s-1}^{2}).
\label{3012}
\end{equation}
Therefore, we get
\begin{equation}
\frac{d}{dt}(\|u\|_{s}^{2}+\|\rho\|_{s-1}^{2})\leq (C+\lambda)(\|u\|_{s}^{2}+\|\rho\|_{s-1}^{2}).
\label{3013}
\end{equation}
Applying Gronwall's inequality, we have
$$\|u\|_{s}^{2}+\|\rho\|_{s-1}^{2}\leq e^{(C+\lambda)t}(\|u_{0}\|_{s}^{2}+\|\rho_{0}\|_{s-1}^{2}).$$
This completes the proof of Theorem \ref{the301}.
\end{proof}

Inspired by Constantin's work \cite{Constantin2000}, we consider the following trajectory equation related to system \eqref{202}
\begin{equation}
\begin{cases}
\frac{dq(t,x)}{dt}=k u^{2}(t,q(t,x)),\ (t,x)\in [0,T)\times \mathbb{R},\\
q(0,x)=x,\ x\in \mathbb{R},
\end{cases}
\label{3014}
\end{equation}
where $u$ denotes the first component of the solution $z$ to \eqref{202}. Applying the classical results in the theory of ODEs, we can obtain the following results on $q$ which is crucial in the proof of blowup scenarios.

Similar to Lemma 3.7 in \cite{Wu2012}, we can give the following lemma but omit the detailed proof.
\begin{lemma}
\label{lem304} Let $u\in C([0,T); H^{s})\cap C^{1}([0,T); H^{s-1}), s>3/2$. Then \eqref{3014} has a unique solution $q(t,x)\in C^{1}([0,T)\times \mathbb{R}, \mathbb{R})$, and $q(t,\cdot)$ is an increasing diffeomorphism of $\mathbb{R}$ with
$$q_{x}(t,x)=e^{\int_{0}^{t}2ku(\tau,q(\tau,x))u_{x}(\tau,q(\tau,x))d\tau}>0,\ (t,x)\in [0,T)\times \mathbb{R}.$$
\end{lemma}

\begin{lemma}
\label{lem305} Let $z_{0}\in H^{s}\times H^{s-1} (s>\frac{3}{2})$, and $T>0$ be the maximal existence time of the corresponding solution $z=\left(\begin{array}{cc}u\\ \rho\end{array}\right)$ of system \eqref{202}.
If there exists $N>0$ such that $kuu_{x}(t,x)\geq-N$ for all $(t,x)\in[0,T)\times \mathbb{R}$, then
\begin{equation*}
\|\rho(t,\cdot)\|_{L^{\infty}}=\|\rho(t,q(t,\cdot))\|_{L^{\infty}}\leq e^{NT}\|\rho_{0}(\cdot)\|_{L^{\infty}}.
\end{equation*}
\end{lemma}
\begin{proof}
By \eqref{3014} and \eqref{202}, we get
\begin{align*}
\frac{d}{dt}\rho(t,q(t,x))
=\rho_{t}+\rho_{x}q_{t}=\rho_{t}+ku^{2}\rho_{x}=-k\rho uu_{x},
\end{align*}
which leads to
$$\rho(t,q(t,x))=\rho_{0}e^{-k\int_{0}^{t}uu_{x}(\tau,q(\tau,x))d\tau}.$$
If there exists a positive constant $N>0$ such that $kuu_{x}\geq-N$, then we get
\begin{equation*}
\|\rho(t,\cdot)\|_{L^{\infty}}=\|\rho(t,q(t,\cdot))\|_{L^{\infty}}
\leq\|\rho_{0}(\cdot)\|_{L^{\infty}} e^{NT}<\infty.
\end{equation*}
This completes the proof of Lemma \ref{lem302}.
\end{proof}

\begin{theorem}
\label{the302} Let $z_{0}:=\left(
\begin{array}{cc} u_{0}\\ \rho_{0}\end{array}\right)\in H^{2}\times H^{1}$ be given, and $T$ be the maximal existence time of the corresponding solution $z=\left(\begin{array}{cc}u\\ \rho\end{array}\right)$ of \eqref{202} with the initial data $z_{0}$.
Then $T$ is finite if and only if
$$
\liminf_{t\rightarrow T}\inf_{x\in\mathbb{R}}{kuu_{x}}(t,x)=-\infty.
$$
\end{theorem}
\begin{proof}
By Theorem \ref{the201}, for any initial data $z_{0}\in H^{2}\times H^{1}$ which corresponding $s=2$, there exists a unique local solution $z$ of system \eqref{202}.
The following study will focus on investigating the necessary and sufficient conditions for the occurrence of wave breaking in finite time.

Note that $y=u-u_{xx}$, we get
\begin{equation*}
\|u\|_{H^{2}}^{2}\leq \|y\|_{L^{2}}^{2}=\int_{\mathbb{R}}(u-u_{xx})^{2}dx
=\int_{\mathbb{R}}(u^{2}+u_{xx}^{2}+2u_{x}^{2})dx,
\end{equation*}
where we used the fact that
$
\int_{\mathbb{R}}2uu_{xx}dx=-2\int_{\mathbb{R}}u_{x}^{2}dx.
$

Multiplying the first equation in \eqref{303} by $2y$, we get
\begin{equation}
2yy_{t}+2k u^{2}yy_{x}+6k uu_{x}y^{2}+2ky\rho(u\rho)_{x}=-2\left(g(u)-\frac{4k}{3}u^{3}\right)_{x}y-2\lambda y^{2}.
\label{3015}
\end{equation}
Integrating \eqref{3015} with respect to\ $x$ over $\mathbb{R}$, we obtain
\begin{align*}
\frac{d}{dt}\int_{\mathbb{R}}y^{2}dx=&-4k \int_{\mathbb{R}}uu_{x}y^{2}dx
-2\lambda\int_{\mathbb{R}}y^{2}dx-2\int_{\mathbb{R}}
\left(g(u)-\frac{4k}{3}u^{3}\right)_{x}ydx\nonumber\\
&+2k\int_{\mathbb{R}}u_{x}u_{xx}\rho^{2}dx+2k\int_{\mathbb{R}}uu_{xx}\rho \rho_{x}dx.
\label{5022}
\end{align*}
Multiplying the second equation in \eqref{303} by $2\rho$, we get
\begin{equation}
2\rho\rho_{t}+2ku^{2}\rho\rho_{x}=-2k\rho^{2}uu_{x}.
\label{3016}
\end{equation}
Integrating \eqref{3016} with respect to\ $x$ over $\mathbb{R}$, we obtain
\begin{equation*}
\frac{d}{dt}\int_{\mathbb{R}}\rho^{2}dx=0.
\label{5022}
\end{equation*}
Differentiating the second equation in \eqref{303} with respect to $x$, we get
\begin{equation}
\rho_{tx}+3kuu_{x}\rho_{x}+ku^{2}\rho_{xx}+k\rho u_{x}^{2}+kuu_{xx}\rho=0.
\label{3017}
\end{equation}
Multiplying \eqref{3017} by $2\rho_{x}$, we get
\begin{equation}
2\rho_{x}\rho_{tx}+6kuu_{x}\rho_{x}^{2}+2ku^{2}\rho_{x}\rho_{xx}+2k\rho\rho_{x} u_{x}^{2}+2kuu_{xx}\rho\rho_{x}=0.
\label{3018}
\end{equation}
Integrating \eqref{3018} with respect to $x$, we get
\begin{align*}
\frac{d}{dt}\int_{\mathbb{R}}\rho_{x}^{2}dx
=&-4k\int_{\mathbb{R}}uu_{x}\rho_{x}^{2}dx-2k\int_{\mathbb{R}}u^{2}_{x}\rho\rho_{x}dx
-2k\int_{\mathbb{R}}uu_{xx}\rho\rho_{x}dx\nonumber\\
=&-4k\int_{\mathbb{R}}uu_{x}\rho_{x}^{2}dx+2k\int_{\mathbb{R}}u_{x}u_{xx}\rho^{2}dx
-2k\int_{\mathbb{R}}uu_{xx}\rho\rho_{x}dx.
\label{5022}
\end{align*}
Therefore, we get
\begin{align*}
\frac{d}{dt}\int_{\mathbb{R}}(y^{2}+\rho^{2}+\rho_{x}^{2})dx
=&-4k\int_{\mathbb{R}}uu_{x}y^{2}dx
-2\lambda\int_{\mathbb{R}}y^{2}dx-2\int_{\mathbb{R}}
\left(g(u)-\frac{4k}{3}u^{3}\right)_{x}ydx\nonumber\\
&+4k\int_{\mathbb{R}}u_{x}u_{xx}\rho^{2}dx-4k\int_{\mathbb{R}}uu_{x}\rho_{x}^{2}dx.
\end{align*}
Note that
\begin{align*}
2\int_{\mathbb{R}}\left(g(u)-\frac{4k}{3}u^{3}\right)_{x}ydx
\leq&\int_{\mathbb{R}}\left(\left(g(u)-\frac{4k}{3}u^{3}\right)_{x}^{2}+y^{2}\right)dx
\nonumber\\
\leq&\left\|g(u)-\frac{4k}{3}u^{3}\right\|_{1}^{2}+\|y\|_{L^{2}}^{2}\nonumber\\
\leq& C_{1}\|u\|_{1}^{2}+\|y\|_{L^{2}}^{2}\nonumber\\
\leq& C_{2}\|y\|_{L^{2}}^{2},
\end{align*}
where $C_{1}, C_{2}>0$ represent specific constants.

If $k uu_{x}(t,x)$ is bounded from below on $[0,T)\times \mathbb{R}$, i.e., there exists $N>0$ such that $k uu_{x}(t,x)>-N$ on $[0,T)\times \mathbb{R}$. Then we get
\begin{align*}
\frac{d}{dt}\int_{\mathbb{R}}(y^{2}+\rho^{2}+\rho_{x}^{2})dx
\leq&(4N+2\lambda+4|k|e^{NT}\|\rho_{0}(\cdot)\|_{L^{\infty}}^{2}+c_{2})
\int_{\mathbb{R}}(y^{2}+\rho^{2}+\rho_{x}^{2})dx\nonumber\\
\leq&(C+2\lambda)\int_{\mathbb{R}}(y^{2}+\rho^{2}+\rho_{x}^{2})dx.
\end{align*}
Applying Gronwall's inequality, we get
\begin{equation*}
\int_{\mathbb{R}}(y^{2}+\rho^{2}+\rho_{x}^{2})dx\leq e^{\left(C+ 2\lambda\right)t}\int_{\mathbb{R}}(y_{0}^{2}+\rho_{0}^{2}+\rho_{0x}^{2})dx.
\label{5026}
\end{equation*}
It then follows from the Sobolev embedding theorem $H^{2}\hookrightarrow L^{\infty}$, we get
\begin{equation*}
\|u\|_{2}^{2}+\|\rho\|_{1}^{2}\leq \|y\|_{L^{2}}^{2}+\|\rho\|_{1}^{2}\leq (\|y_{0}\|_{L^{2}}^{2}+\|\rho_{0}\|_{1}^{2})e^{\left(C+2\lambda\right)t}.
\end{equation*}
By using Theorem \ref{the301}, we deduce that the solution exists globally in time.

On the other hand, if $kuu_{x}(t,x)$ is unbounded from below, by applying Theorem \ref{the301} and using the Sobolev embedding theorem $H^{s}\hookrightarrow L^{\infty}$ ($s>\frac{1}{2})$, we infer that the solution will blow up in finite time.
This completes the proof of Theorem \ref{the302}.
\end{proof}

\section{Persistence property}

\setcounter{equation}{0}

\label{sec:4}

In this section, we mainly discuss the persistence property for a weakly dissipative generalized 2-component Novikov system.
To this end, we reformulate system \eqref{202} as follows
\begin{equation}
\begin{cases}
u_{t}+k u^{2}u_{x}=-P\ast A(u,\rho)-\partial_{x}P\ast B(u,\rho)-\lambda u,\\
\rho_{t}+k u^{2}\rho_{x}=-k \rho uu_{x},
\label{401}
\end{cases}
\end{equation}
where $A(u,\rho)=\frac{k}{2}\left(u_{x}^{3}-u_{x}\rho^{2}\right)$ and $B(u,\rho)=g(u)+\frac{3k}{2}uu_{x}^{2}-\frac{k}{3}u^{3}
-\frac{k}{2}u\rho^{2}$.

Before giving the theorem on the persistence of solutions, we first introduce the definition for admissible weight function.
\begin{definition}[\cite{Brandolese2012}]
\label{def401} An admissible weighted function for system \eqref{401} is locally absolutely function $\phi: \mathbb{R} \to \mathbb{R}$ such that for some $\theta> 0$ and a.e. $x \in \mathbb{R}$, $|\phi'(x)|\leq \theta|\phi(x)|$, and that is f-moderate for some sub-multiplicative weight function $f$ satisfying
$$\inf_{\mathbb{R}}f>0\ \ \text{and}\ \ \int_{\mathbb{R}} \frac{f(x)}{e^{|x|}} \mathrm{d}x<\infty.$$
\end{definition}

In what follows, we state the main theorem of this section.
\begin{theorem}\label{the401}
Let $s>\frac{3}{2}$ and $2\leq p\leq\infty.$ Let $(u,\rho)\in\mathcal{C}([0,T),H^s(\mathbb{R})\times H^{s-1}(\mathbb{R}))\cap\mathcal{C}^1([0,T), H^{s-1}(\mathbb{R})\times H^{s-2}(\mathbb{R}))$ be a strong solution of system \eqref{401} such that $(u_0,\rho_0)$ satisfies $u_0\psi,u_{0x}\psi$ and $\rho_0\psi\in L^p(\mathbb{R}),$ where $\psi$ is an admissible weighted function for system \eqref{401}. Then, for all $t\in[0,T)$, there exists a constant $C>0$ depending only on the weight function $\psi$ such that
\begin{align}
\|u(t)\psi\|_{L^{p}}+\|u_{x}(t)\psi\|_{L^{p}}+\|\rho(t)\psi\|_{L^{p}}
\leq e^{(CM^{2}+\lambda)t}(\|u_{0}\psi\|_{L^{p}}+\|u_{0x}\psi\|_{L^{p}}+\|\rho_{0}\psi\|_{L^{p}}),
\label{402}
\end{align}
where
$$M\equiv\sup_{t\in[0,T)}(\|u(t)\|_{L^{\infty}}+\|u_{x}(t)\|_{L^{\infty}}+\|\rho(t)\|_{L^{\infty}})<\infty.$$
\end{theorem}

The fundamental class of instances for admissible weighted functions is provided by the family of functions.
$$\psi_{a,b,c,d}=e^{a|x|^{b}}(1+|x|^{c})(\log(e+|x|)^{d}),$$
where $\psi$ must satisfy the following conditions
\begin{equation*}
a\geq0,\ \ 0\leq b\leq1,\ \ ab<1,\ \ c,d\in \mathbb{R}.
\end{equation*}

\begin{definition}[\cite{Brandolese2012}]
\label{def402} A function $f:\mathbb{R}\rightarrow\mathbb{R}$ called sub-multiplicative if
$$f(x+y)\leq f(x)f(y),\ \ \text{for all}\ \ x, y \in \mathbb{R}.$$
Given a non-negative sub-multiplicative function $f$, a function $\psi$ is called $f$-moderate if and only if
$$\exists C_0> 0:\psi(x+y)\leq C_0 f(x) \psi(y),$$
for all $x, y \in \mathbb{R}$. A function $\psi$ is called moderate if and only if there exists a nonnegative sub-multiplicative function $f$ such that $\psi$ is $f$-moderate.
\end{definition}

Let us recall the most standard examples:
$$\psi_{a,b,c,d}(x)= e^{a|x|^b}(1+|x|^c)\left(\log(e+|x|)^d\right).$$
We have the following conditions (see \cite{Brandolese2012}):
\begin{description}
\item[$(i)$] For $a, c, d > 0$ and $0 \leq b \leq 1$, the function $\phi_{a,b,c,d}(x)$ is sub-multiplicative.
\item[$(ii)$] If $a, c, d \in \mathbb{R}$ and $0 \leq b \leq 1$, then $\phi_{a,b,c,d}(x)$ is moderate, that is $\phi_{a,b,c,d}$ is $\psi_{\alpha,\beta,\gamma,\delta}$-moderate for $|a|\leq\alpha, |b|\leq\beta, |c|\leq\gamma, |d|\leq\delta$.
\end{description}

\textbf{Proof of Theorem 4.1.} We first introduce a auxiliary function as follows
\begin{equation*}
\psi_{N}(x)=
\begin{cases}
\psi(x),\ \ \ \psi(x)\leq N,\\
N,\ \ \ \ \ \ \ \psi(x)>N,
\end{cases}
\end{equation*}
where $N\in \mathbb{N}$. For all $N$, we have
\begin{equation*}
\|\psi_{N}(x)\|_{L^{\infty}}\leq N,\ \ 0<|\psi_{N}'(x)|\leq \theta|\psi_{N}(x)|\ \ \text{and}\ \ \psi_{N}(x+y)\leq C_{1}f(x)\psi_{N}(y),
\end{equation*}
where $C_{1}=\max\{C_{0},\alpha^{-1}\}, C_{0}$ is given in Definition 4.2, $\alpha=\inf_{x\in \mathbb{R}}f(x)>0.$

In what follows, we consider the case $2\leq p<\infty$. Multiplying both sides of the first equation of system \eqref{401} by $\psi_{N}|u\psi_{N}|^{p-2}(u\psi_{N})$ and integrating the resulting equation with respect to $x$ over $\mathbb{R}$, we can obtain
\begin{align}
&\int_{\mathbb{R}}(u_{t}\psi_{N})|u\psi_{N}|^{p-2}(u\psi_{N})dx+k\int_{\mathbb{R}}(uu_{x})|u\psi_{N}|^{p}dx\nonumber\\
=&-\int_{\mathbb{R}}(\psi_{N}\cdot P\ast A(u,\rho))|u\psi_{N}|^{p-2}(u\psi_{N})dx
-\int_{\mathbb{R}}(\psi_{N}\cdot P_{x}\ast B(u,\rho))|u\psi_{N}|^{p-2}(u\psi_{N})dx\nonumber\\
&-\lambda\int_{\mathbb{R}}|u\psi_{N}|^{p}dx.
\label{403}
\end{align}


Noticing that
\begin{align}
\int_{\mathbb{R}}(u_{t}\psi_{N})|u\psi_{N}|^{p-2}(u\psi_{N})dx
=\frac{1}{p}\frac{d}{dt}\|u\psi_{N}\|_{L^{p}}^{p}
=\|u\psi_{N}\|_{L^{p}}^{p-1}\frac{d}{dt}\|u\psi_{N}\|_{L^{p}},
\label{404}
\end{align}
\begin{equation}
\int_{\mathbb{R}}(uu_{x})|u\psi_{N}|^{p}dx
\leq\|uu_{x}\|_{L^{\infty}}\|u\psi_{N}\|_{L^{p}}^{p}\leq M^{2}\|u\psi_{N}\|_{L^{p}}^{p},
\label{405}
\end{equation}
\begin{align}
&\left|\int_{\mathbb{R}}(\psi_{N}\cdot P\ast A(u,\rho))|u\psi_{N}|^{p-2}(u\psi_{N})dx\right|\nonumber\\
\leq&\|u\psi_{N}\|_{L^{p}}^{p-1}\|\psi_{N}\cdot P\ast A(u,\rho)\|_{L^{p}}\nonumber\\
\leq& C_{2}\|u\psi_{N}\|_{L^{p}}^{p-1}\|P\ast f\|_{1}\left\|\psi_{N}\left(\frac{k}{2}(u_{x}^{3}-u_{x}\rho^{2})\right)\right\|_{L^{p}}\nonumber\\
\leq& C_{3}M^{2}\|u\psi_{N}\|_{L^{p}}^{p-1}\left(\left\|u_{x}\psi_{N}\right\|_{L^{p}}+\left\|\rho\psi_{N}\right\|_{L^{p}}\right).
\label{407}
\end{align}
Similarly, we can estimate
\begin{align}
&\left|\int_{\mathbb{R}}(\psi_{N}\cdot P_{x}\ast B(u,\rho))|u\psi_{N}|^{p-2}(u\psi_{N})dx\right|\nonumber\\
\leq&C_{4}M^{2}\|u\psi_{N}\|_{L^{p}}^{p-1}\left(\left\|u\psi_{N}\right\|_{L^{p}}
+\left\|u_{x}\psi_{N}\right\|_{L^{p}}+\left\|\rho\psi_{N}\right\|_{L^{p}}\right).
\label{407}
\end{align}
\begin{equation}
\int_{\mathbb{R}}|u\psi_{N}|^{p}dx=\|u\psi_{N}\|_{L^{p}}^{p},
\label{408}
\end{equation}
where $C_{2}-C_{4}$ independent of $N$, $M=\sup_{t\in[0,T)}(\|u\|_{L^{\infty}}+\|u_{x}\|_{L^{\infty}}+\|\rho\|_{L^{\infty}})<\infty.$

Combining \eqref{403}-\eqref{408}, we get
\begin{equation}
\frac{d}{dt}\|u\psi_{N}\|_{L^{p}}\leq(C_{5}M^{2}+\lambda)\left(\left\|u\psi_{N}\right\|_{L^{p}}
+\left\|u_{x}\psi_{N}\right\|_{L^{p}}+\left\|\rho\psi_{N}\right\|_{L^{p}}\right).
\label{409}
\end{equation}

In what follows, we will give estimates on $u_{x}\psi_{N}$. Differentiating first equation of system \eqref{401} with respect to $x$, we get
\begin{equation}
u_{tx}+ku^{2}u_{xx}+2kuu_{x}^{2}=-P_{x}\ast A(u,\rho)-P_{xx}\ast B(u,\rho)-\lambda u_{x}.
\label{4010}
\end{equation}
Multiplying \eqref{4010} with $\psi_{N}|u_{x}\psi_{N}|^{p-2}(u_{x}\psi_{N})$ and integrating the resulting equation with respect to $x$ over $\mathbb{R}$, we can obtain
\begin{align}
&\int_{\mathbb{R}}\partial_{t}(u_{x}\psi_{N})|u_{x}\psi_{N}|^{p-2}(u_{x}\psi_{N})dx
+k\int_{\mathbb{R}}u^{2}u_{xx}\psi_{N}|u_{x}\psi_{N}|^{p-2}(u_{x}\psi_{N})dx\nonumber\\
&+2k\int_{\mathbb{R}}uu^{2}_{x}\psi_{N}|u_{x}\psi_{N}|^{p-2}(u_{x}\psi_{N})dx\nonumber\\
=&-\int_{\mathbb{R}}(\psi_{N}\cdot P_{x}\ast A(u,\rho))|u_{x}\psi_{N}|^{p-2}(u_{x}\psi_{N})dx
-\int_{\mathbb{R}}(\psi_{N}\cdot P_{xx}\ast B(u,\rho))|u_{x}\psi_{N}|^{p-2}(u_{x}\psi_{N})dx\nonumber\\
&-\lambda\int_{\mathbb{R}}|u_{x}\psi_{N}|^{p}dx.
\label{4011}
\end{align}
In what follows, we estimate the terms in\eqref{4011}.
\begin{align}
\int_{\mathbb{R}}\partial_{t}(u_{x}\psi_{N})|u_{x}\psi_{N}|^{p-2}(u_{x}\psi_{N})dx
=\|u_{x}\psi_{N}\|_{L^{p}}^{p-1}\frac{d}{dt}\|u_{x}\psi_{N}\|_{L^{p}},
\label{4012}
\end{align}
\begin{align}
&\left|\int_{\mathbb{R}}u^{2}u_{xx}\psi_{N}|u_{x}\psi_{N}|^{p-2}(u_{x}\psi_{N})dx\right|\nonumber\\
=&\left|\int_{\mathbb{R}}\left(\partial_{x}(u_{x}\psi_{N})-u_{x}\psi_{Nx}\right)u^{2}|u_{x}\psi_{N}|^{p-2}(u_{x}\psi_{N})dx\right|\nonumber\\
=&\left|\frac{1}{p}\int_{\mathbb{R}}u^{2}\left(\partial_{x}|u_{x}\psi_{N}|^{p}\right)|dx
-\int_{\mathbb{R}}|u_{x}\psi_{N}|^{p-2}(u_{x}\psi_{N})u^{2}u_{x}\psi_{Nx}dx\right|\nonumber\\
\leq&\frac{2M^{2}}{p}\|u_{x}\psi_{N}\|_{L^{p}}^{p}+\theta M^{2}\|u_{x}\psi_{N}\|_{L^{p}}^{p},
\label{4013}
\end{align}
\begin{equation}
\left|\int_{\mathbb{R}}uu^{2}_{x}\psi_{N}|u_{x}\psi_{N}|^{p-2}(u_{x}\psi_{N})dx\right|\leq 2 M^{2}\|u_{x}\psi_{N}\|_{L^{p}}^{p},
\label{4014}
\end{equation}
\begin{align}
&\left|\int_{\mathbb{R}}(\psi_{N}\cdot P_{x}\ast A(u,\rho))|u_{x}\psi_{N}|^{p-2}(u_{x}\psi_{N})dx\right|\nonumber\\
\leq&\|u_{x}\psi_{N}\|_{L^{p}}^{p-1}\|\psi_{N}\cdot P_{x}\ast A(u,\rho)\|_{L^{p}}\nonumber\\
\leq& C_{6}M^{2}\|u_{x}\psi_{N}\|_{L^{p}}^{p-1}\left(\left\|u_{x}\psi_{N}\right\|_{L^{p}}+\left\|\rho\psi_{N}\right\|_{L^{p}}\right),
\label{4015}
\end{align}
\begin{align}
&\left|\int_{\mathbb{R}}(\psi_{N}\cdot P_{xx}\ast B(u,\rho))|u_{x}\psi_{N}|^{p-2}(u_{x}\psi_{N})dx\right|\nonumber\\
\leq&C_{7}M^{2}\|u_{x}\psi_{N}\|_{L^{p}}^{p-1}\left(\left\|u\psi_{N}\right\|_{L^{p}}
+\left\|u_{x}\psi_{N}\right\|_{L^{p}}+\left\|\rho\psi_{N}\right\|_{L^{p}}\right).
\label{4016}
\end{align}
\begin{equation}
\int_{\mathbb{R}}|u_{x}\psi_{N}|^{p}dx=\|u_{x}\psi_{N}\|_{L^{p}}^{p},
\label{4017}
\end{equation}
Combining \eqref{4011}-\eqref{4017}, we obtain
\begin{equation}
\frac{d}{dt}\|u_{x}\psi_{N}\|_{L^{p}}\leq(C_{8}M^{2}+\lambda)\left(\left\|u\psi_{N}\right\|_{L^{p}}
+\left\|u_{x}\psi_{N}\right\|_{L^{p}}+\left\|\rho\psi_{N}\right\|_{L^{p}}\right).
\label{4018}
\end{equation}

In what follows, we multiply both sides of the second equation of \eqref{401} by $\psi_{N}|\rho\psi_{N}|^{p-2}(\rho\psi_{N})$ and integrating the resulting equation with respect to $x$ over $\mathbb{R}$, we get
\begin{align}
&\int_{\mathbb{R}}(\rho_{t}\psi_{N})|\rho\psi_{N}|^{p-2}(\rho\psi_{N})dx
+k\int_{\mathbb{R}}(u^{2}\rho_{x}\psi_{N})|\rho\psi_{N}|^{p-2}(\rho\psi_{N})dx\nonumber\\
=&-k\int_{\mathbb{R}}(uu_{x}\rho\psi_{N})|\rho\psi_{N}|^{p-2}(\rho\psi_{N})dx.
\label{4019}
\end{align}
In what follows, we estimate the terms in\eqref{4019}.
\begin{equation}
\int_{\mathbb{R}}(\rho_{t}\psi_{N})|\rho\psi_{N}|^{p-2}(\rho\psi_{N})dx
=\|\rho\psi_{N}\|_{L^{p}}^{p-1}\frac{d}{dt}\|\rho\psi_{N}\|_{L^{p}},
\label{4020}
\end{equation}
\begin{align}
&\left|\int_{\mathbb{R}}(u^{2}\rho_{x}\psi_{N})|\rho\psi_{N}|^{p-2}(\rho\psi_{N})dx\right|\nonumber\\
=&\left|\int_{\mathbb{R}}(u^{2}(\partial_{x}(\rho\psi_{N})-\rho\psi_{Nx})|\rho\psi_{N}|^{p-2}(\rho\psi_{N})dx\right|\nonumber\\
=&\left|-\frac{2}{p}uu_{x}|\rho\psi_{N}|^{p}dx-\int_{\mathbb{R}}(u^{2}\rho\psi_{Nx})|\rho\psi_{N}|^{p-2}(\rho\psi_{N})dx\right|\nonumber\\
\leq&\frac{2M^{2}}{p}\|\rho\psi_{N}\|_{L^{p}}^{p}+\theta M^{2}\|\rho\psi_{N}\|_{L^{p}}^{p}\nonumber\\
\leq&(1+\theta)M^{2}\|\rho\psi_{N}\|_{L^{p}}^{p},
\label{4021}
\end{align}
\begin{equation}
\left|\int_{\mathbb{R}}(uu_{x}\rho\psi_{N})|\rho\psi_{N}|^{p-2}(\rho\psi_{N})dx\right|\leq M^{2}\|\rho\psi_{N}\|_{L^{p}}^{p}.
\label{4022}
\end{equation}
From \eqref{4019}-\eqref{4022}, we obtain
\begin{equation}
\frac{d}{dt}\|\rho\psi_{N}\|_{L^{p}}\leq C_{9}M^{2}\left\|\rho\psi_{N}\right\|_{L^{p}}.
\label{4023}
\end{equation}

Combining \eqref{409}, \eqref{4018} and \eqref{4023}, we obtain
\begin{equation}
\frac{d}{dt}\left(\|u\psi_{N}\|_{L^{p}}+\|u_{x}\psi_{N}\|_{L^{p}}+\|\rho\psi_{N}\|_{L^{p}}\right)
\leq(CM^{2}+\lambda)\left(\left\|u\psi_{N}\right\|_{L^{p}}
+\left\|u_{x}\psi_{N}\right\|_{L^{p}}+\left\|\rho\psi_{N}\right\|_{L^{p}}\right),
\label{4024}
\end{equation}
where $C$ is a constant only depending on $\theta$ and $f$.

By using Gronwall's inequality, we get
\begin{equation}
\|u\psi_{N}\|_{L^{p}}+\|u_{x}\psi_{N}\|_{L^{p}}+\|\rho\psi_{N}\|_{L^{p}}
\leq e^{(CM^{2}+\lambda)t}(\|u_{0}\psi_{N}\|_{L^{p}}+\|u_{0x}\psi_{N}\|_{L^{p}}+\|\rho_{0}\psi_{N}\|_{L^{p}}).
\label{4025}
\end{equation}
Taking advantage of Lebesgue's dominated convergence theorem and taking the limit as $N$ goes to $\infty$, we get
\begin{equation}
\|u\psi\|_{L^{p}}+\|u_{x}\psi\|_{L^{p}}+\|\rho\psi\|_{L^{p}}
\leq e^{(CM^{2}+\lambda)t}(\|u_{0}\psi\|_{L^{p}}+\|u_{0x}\psi\|_{L^{p}}+\|\rho_{0}\psi\|_{L^{p}}).
\label{4025}
\end{equation}
It is obvious that the theorem also applies for $p=\infty$, since $\|\cdot\|_{L^{\infty}}=\lim_{p\rightarrow \infty}\|\cdot\|_{L^{p}}$.

This completes the proof of Theorem \ref{the401}.

{\bf Acknowledgements.} This work is partially supported by NSFC Grants (nos. 12201539 and 12225103), Natural Science Foundation of Xinjiang Uygur Autonomous Region (no. 2022D01C65), Natural Science Foundation of Gansu Province (no. 23JRRG0006), Youth Doctoral Support Project for Universities in Gansu Province (no. 2024QB-106), Innovation Fund for University Teachers of Gansu Province (no. 2024A-149), Doctoral Fund of Hexi University (no. KYQD2025006), and the NCST distinguished collaborative program.

\bibliographystyle{elsarticle-num}
\bibliography{<your-bib-database>}



\end{document}